\documentclass[12pt,a4paper]{article}
\usepackage{amsmath, amsthm, amscd, amsfonts, amssymb, graphicx, color}
\usepackage{accents}
\usepackage[bookmarksnumbered, colorlinks, plainpages]{hyperref}

\makeatletter
\@namedef{subjclassname@2010}{%
	\textup{2010} Mathematics Subject Classification}
\makeatother
\usepackage[mathscr]{euscript}
\usepackage{latexsym}
\usepackage{multicol,tabularx}
\usepackage{tikz}
\usepackage{graphics}
\usepackage{mathtools}

\usepackage{amsfonts}
\usepackage{amssymb}

\usepackage[english]{babel}
\usepackage{multido}
\usepackage[nomessages]{fp}
\usepackage[margin=2.5cm]{geometry}
\usepackage{enumitem}
\theoremstyle{plain}
\newtheorem{lemma}{Lemma}[section]

\newtheorem{coro*}{Corollary}[section]
\newtheorem{observation}{Observation}[section]
\newtheorem{theorem}{Theorem}[section]

\theoremstyle{definition}

\newtheorem{defn}{Definition}[section]

\newtheorem*{res*}{Theorem}
\numberwithin{equation}{section}

\begin{document}
\title{An Edge Labeling of Graphs from Rado's Partition Regularity Condition}

\author {Arun J Manattu\footnote{E-mail: arunjmanattu@gmail.com} ~ and Aparna Lakshmanan S\footnote{E-mail: aparnals@cusat.ac.in, aparnaren@gmail.com}\\
Department of Mathematics\\	Cochin University of Science and Technology\\Cochin -
	22}
\maketitle
\begin{abstract}
A vertex $v$ is called an AR-vertex, if $v$ has distinct edge weight sums for each distinct subset of edges incident on $v$. i.e., if $\{x_1,x_2,\dots,x_k\}$ are the edge labels of the edges incident on $v$, then the $2^k$ subset sums are all distinct. An injective edge labeling $f$ of a graph $G$ is said to be an AR-labeling of $G$, if $f:E \rightarrow \mathbb{N}$ is such that every vertex in $G$ is an AR-vertex under $f$. A graph $G$ is said to be an AR-graph, if there exists an AR-labeling $f:E\rightarrow \{1,2,\dots,m\}$, where $m$ denotes the number of edges of $G$. A study of AR-labeling and AR-graphs is initiated in this paper.\\
\noindent \line(1,0){395}

\noindent \line(1,0){395}\\
\end{abstract}

\section{Introduction}
The Ramsey theory is a branch of combinatorics exploring the presence of order in every given structure man constructs. Philosophically speaking, the Ramsey Theory states that there cannot be absolute chaos in any system. A precursor to this idea can be found in Schur’s Theorem \cite{Rob} by Issai Schur which states that, if the set of positive integers $\mathbb{N}$ is finitely colored then there exists $x,y,z$ having the same color such that $x+y=z.$\par
Richard Rado, a student of Schur, proved that the equation $c_{1}x_{1}+c_{2}x_{2}+\dots +c_{n}x_{n}=0$ is partition regular (i.e., it has a monochromatic solution whenever $\mathbb{N}$ is finitely colored) if and only if there is $J\subseteq \{1,2,3,\dots,n\}$ such that $\sum_{i \in J} c_{i}=0$. Schur’s Theorem could be identified as a particular case of Rado’s theorem with coefficients of $x_{1}, x_{2}$ and $x_{3}$ being 1, 1 and -1, respectively. Rado’s theorem is considered a part of the Ramsey theory, since it assures that, given any linear equation that satisfies the regularity condition, in any $r$-coloring of sufficiently large natural numbers, there will be a monochromatic solution. So, the regularity condition of Rado enforced on the linear equation is powerful and cannot be compromised.\par
Motivated by the strength of this condition, if we introduce a similar line of restriction on the vertices in an edge labelled graph, soon we realize that the number-theoretic notion of regularity slowly disappears while using sufficiently large natural numbers. In other words, given any graph, we can label its edges with natural numbers (sufficiently large) so that the number-theoretic notion of regularity is absent from every vertex of the graph. In fact, given a graph $G$, if we restrict our labeling to a one-one function from the edge set $E(G)$ of the graph $G$ to $\mathbb{N}$, vertices with degree at most 2 will never be number-theoretically regular. So we asked the question, can we label the edges of a graph in such a way that the set of edge weights of a given vertex has distinct subset sums? In other words, in terms of Rado's partition regularity, can we label edges so that among the weights of edges incident on a vertex, there exists no linear combination of them (with coefficients $\pm 1$) summing to 0? In this paper, we formulate this problem as an edge labeling problem in graphs which we termed as AR-labeling\footnote{The labeling is named AR since the letters A and R are the common letters in the names of the authors and Rado, whose partition regularity condition was the motivation behind the idea of this labeling.}.\par
By a graph $G=(V,E)$ we mean a finite simple undirected graph. The order $\left|V\right|$ and
the size $\left|E\right|$ of $G$ are denoted by $n$ and $m$, respectively.
For all graph theoretic terminology and notations not mentioned here, we refer to Balakrishnan and Ranganathan  \cite{1}.
\subsection{New Definitions and Terminology}
\begin{defn}
    Let $f: E \rightarrow \mathbb{N}$ be an injective edge labeling of a graph $G$. A vertex $v$ is called an AR-vertex, if $v$ has distinct edge weight sums for each distinct subset of edges incident on $v$. i.e., if $\{x_1,x_2,\dots,x_k\}$ are the edge labels of the edges incident on $v$, then the $2^k$ subset sums are all distinct.
\end{defn}
\begin{defn}
    An injective edge labeling $f$ of a graph $G$ is said to be an AR-labeling of $G$, if $f:E \rightarrow \mathbb{N}$ is such that every vertex in $G$ is an AR-vertex under $f$.
\end{defn}

\begin{defn}
    A graph $G$ is said to be an AR-graph, if there exists an AR-labeling $f:E \rightarrow \{1,2,\dots,m\}$, where $m$ denotes the number of edges of $G$.
\end{defn}
 
Figure \ref{Figure 2} illustrates AR-labeling of the Petersen graph.
\begin{figure}[ht]\center
  \includegraphics[scale=0.45]{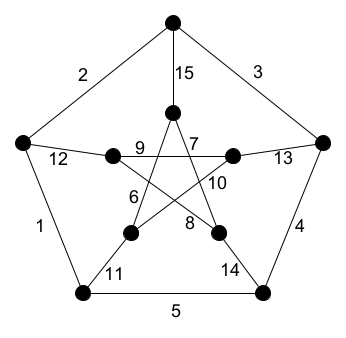} 
  \caption{AR-labeling of Petersen graph}\label{Figure 2}
\end{figure}

AR-labeling of a graph $G$ makes each subset of edges incident on every vertex $v$ unique to the point of identifying them using an ordered pair $(v, k)$, where $k \in \mathbb{N}$. If there is a connected graph representing a communication network with each vertex having only local knowledge, AR-labeling will provide the (server) initiator of communications with distinct commands for each communication that could be translated only by the vertex receiving the command. If labels are determined by the server, the command, even if hacked, cannot be decrypted as long as the individual vertices are not compromised. So, AR-labeling could be used for security networks and defense systems.\par
In an information-theoretic interpretation  \cite{2}, namely in the setting of signaling over a multiple access channel, we can interpret the integers as pulse amplitudes that $n$ transmitters can transmit over an additive channel to send one bit of information each, for example, to signal to the base station that they want to start a communication session. The requirement that all subset sums be distinct expresses the desire that the base station be able to infer any possible subset of active users.

\section{Basic Results on AR-vertices}

In this section, we prove a set of lemmas regarding AR-vertices which will be used to prove theorems in the later sections. We believe that these lemmas will be extremly useful for those who wish to work in this concept.
\begin{lemma}\label{1or2}
    The vertices of $G$ with degree less than or equal to 2 are AR-vertices in every injective edge labeling of $G$.  
\end{lemma}
\begin{proof}
    If $v$ is a vertex with degree one, there is only one edge incident on $v$ and hence only one number appears as an edge weight sum.\\
    If $v$ is a vertex of degree 2. Consider an injective edge labeling of $G$.  Suppose $v$ is adjacent to two edges say $e_{1}$ and $e_{2}$. Let $e_{1}$ and $e_{2}$ be labelled $x$ and $y$ respectively. Due to the injectivity of the labeling, $x\neq y$, and since the labels are from positive integers, $x+y$ cannot be equal to either $x$ or $y$. Hence in both scenarios, $v$ is an AR-vertex. 
\end{proof}

\begin{lemma}\label{3odd}
    A vertex of degree 3 could be made an AR-vertex by labeling the edges incident on $v$ using distinct odd numbers. 
\end{lemma}

\begin{proof}
    Let $x, y$ and $z$ be three distinct odd numbers. The sum of every pair will be an even number, which would not be equal to the third number. And the sum of all three would be greater than each pair-wise sum.
\end{proof}
\begin{lemma}\label{3lab}
    Given a vertex $v$ in G, if any two edges incident on $v$ are labelled $x$ and $y$, then a third edge can be $z$ if and only if $x + y \neq z$ and $|x - y| \neq z$. Moreover, if the edges incident on a vertex $v$ of degree three are labelled $x$, $y$ and $z$ with $x<y<z$, then $v$ is an AR-vertex if $x + y \neq z$.

\end{lemma}

\begin{proof}
    The pair-wise sums including $z$ will always be greater than singletons.
\end{proof}
\begin{lemma}\label{ap}
    If the edges incident on a vertex $v$ with degree three are labelled with numbers in arithmetic progression having common difference $d$, then $v$ is an AR-vertex if the least element $x$ among the edge labels is not $d$, that is $x \neq d$.
\end{lemma}
      
\begin{proof}
    By $Lemma~3$, $v$ is an AR-vertex if and only if $2x+d \neq x+2d$.
\end{proof}
\begin{lemma}\label{4lab}
    Given 8 vertices labelled with some numbers with each number appearing at most
twice, and given 4 distinct numbers $x_1$, $x_2$, $x_3$ and $x_4$, we can always label any set of 4 independent edges between these vertices, using $x_1$, $x_2$, $x_3$ and $x_4$, with the edge labels being distinct from the label of the vertices it is incident on.
\end{lemma}
\begin{proof}
    If the 4 numbers are distinct from the labels, the assignment is trivial. The most
difficult assignment arises when labels and the numbers given are the same. Even if
some labels are different, as long as there are four distinct numbers and each 
label appears at most twice, this proof could be modified and used for the labeling of four
edges between them with the desired property. Rename the vertices as $v_1, v_2, v_3, \dots, v_8$ in such a way that $v_{2i-1}$ and $v_{2i}$ have the same label $x_i$.\\
Consider $v_1$ and $v_2$, they could either be adjacent or not adjacent to each other.\\
{\bf Case 1:} $v_1$ and $v_2$ are adjacent to each other.
\begin{enumerate}
    \item[i] If $v_3$ and $v_4$ are adjacent to each other, label the edges $v_1 v_2$ as $x_3$ and $v_3 v_4$ as $x_4$.
The labeling satisfies our condition irrespective of how we label the remaining
edges. 
\item[ii] If $v_3$ and $v_4$ are adjacent to vertices with the same label, say $v_3$ is adjacent to $v_5$ and $v_4$ is adjacent to $v_6$, label one of them as $x_4$ and the other as $x_1$. Our requirement will be satisfied by the labeling.
\item[iii] If $v_3$ and $v_4$ are adjacent to vertices with different labels, say $v_3$ is adjacent to $v_5$ and $v_4$ is adjacent to $v_7$, label $v_3 v_5$ as $x_4$, $v_4 v_7$ as $x_3$, $v_1 v_2$ as $x_2$ and $v_6 v_8$ as $x_1$.
\end{enumerate}
{\bf Case 2:} $v_1$ and $v_2$ are not adjacent to each other.
\begin{enumerate}
    \item[i] If $v_1$ and $v_2$ have neighbours with the same label, say $v_3$ and $v_4$, label those edges as $x_3$ and $x_4$, then the labeling satisfies our condition.
    \item[ii] If $v_1$ and $v_2$ have neighbours with different labels, say  $v_3$ and $v_5$, label $v_1 v_3$ as $x_3$ and $v_2 v_5$ as $x_4$. Now among the remaining vertices, only one of them has label $x_2$ that is $v_4$. Label the edge incident on $v_4$ as $x_1$ and label the remaining edge as $x_2$. 
\end{enumerate}
So, by exhausting all possibilities, we have proved the lemma.
\end{proof}
Replacing a vertex here with a vertex having a different label that is not already
present does not affect the proof. Furthermore, if we add additional edge
labels with each pair of vertices, this construction could be extended provided at most
two vertices have the same label. For example, if there are 10 vertices and 5 edge
labels, then label any edge between them using a number that is not the label of
any of the vertices it is incident on. The remaining is a set of 8 vertices with at most
two vertices having a specific label and 4 distinct numbers.
\begin{observation} \label{klab}
    As a consequence of Lemma \ref{4lab}, if we have $k$  distinct numbers $(k\geq 4)$ and $2k$ labelled vertices with each
label appearing at most twice, we can always label any set of $k$  independent edges
between these vertices with the edge labels being distinct from the label of the
vertices it is incident on.
\end{observation}

\begin{proof}
    Consider a set of $k$  independent edges and $2k$ labelled vertices with each label appearing at most twice. Start labeling the edges in some order making
sure that the edge label is distinct from the label of the vertices it is incident on.
Proceeding this way label $k - 4$ edges. To label the $i^{th}$ edge, there are $(k - i +1)$
labels. And since $i \leq k-4$, $ k-i \geq 4$. So, the labeling continues without barriers.
Once we label $k - 4$ edges, stop the procedure. What is left behind is 8 vertices with at most two each having same labels and four distinct numbers.
Lemma \ref{4lab} gaurentees the existence of a suitable edge labeling satisfying our condition.
\end{proof}

\begin{lemma}\label{kshift}
    Let $\{a_1, a_2, a_3, a_4\}$ be a set with distinct subset sums with $a_i < a_j$ for $i < j$, then for every $k \geq a_4, \{k+a_1, k+a_2, k+a_3, k+a_4\}$ is also having distinct subset sums. 
\end{lemma}
\begin{proof}
    Since $\{a_1, a_2, a_3, a_4\}$ has 4 distinct elements, so has $\{k+a_1, k+a_2, k+a_3, k+a_4\}$. Now assume that there exists a two-element set with subset sum equalling another two-element set. Because of the ascending order of the elements, this could happen only if $k+a_1 + k+a_4 = k+a_2 + k+a_3$, but this would imply $a_1 + a_4 = a_2 + a_3$ which contradicts the fact that $\{a_1, a_2, a_3, a_4\}$ is a set with distinct subset sums. Hence our assumption is wrong.
    Suppose, if possible, let's assume that the sum of a two-element or three-element subset equals another element, say $k + a_i$ for some $i$. But this implies $a_i > k$ which is a contradiction to our choice of $k$. Hence the lemma.     
\end{proof}
\begin{lemma}\label{newodd}
    If $\{a_1, a_2,\ldots, a_j\}$ is the set of edge labels incident on an AR-vertex with all elements even, then any vertex with degree $(j+1)$ and the set of edge labels incident being $\{a_1, a_2, \ldots, a_j, b\}$ is an AR-vertex if $b$ is an odd number. 
\end{lemma}
\begin{proof}
    No linear combination of the first $j$ labels gives an odd sum. Every linear combination containing the edge label $b$ can be formed only by including $b$. Hence the lemma.  
\end{proof}
\begin{lemma}\label{max10}
    Let $\{a_1, a_2, a_3\}$ be a set with distinct subset sums with $a_i < a_j$ for $i < j$, $\{a_1, a_2, a_3, a_4\}$ is a set with distinct subset sums except for at most 10 values of $a_4$ provided $a_4 \notin \{a_1, a_2, a_3\}$.
\end{lemma}
\begin{proof}
    If $\{a_1, a_2, a_3\}$ has distinct subset sums and $a_4 \neq a_i$ for $i = 1,2,3$, then $\{a_1, a_2, a_3, a_4\}$ will have two distinct subsets with same subset sum if and only if $a_4 \in \{a_2 - a_1, a_3 - a_2, a_3 - a_1, |a_3 - (a_1 + a_2)|, a_3 + a_1 - a_2, a_3 + a_2 - a_1, a_1 + a_2, a_1 + a_3, a_2 + a_3, a_1 + a_2 + a_3 \}$. Hence the lemma.
\end{proof}
\begin{observation}\label{3+even}
    From Lemma \ref{max10}, if $\{a_1, a_2, a_3\}$ is the set of distinct odd numbers and $c$ is an even number greater than all three, then any vertex with a set of edge labels incident being $\{a_1, a_2, a_3, c\}$ is an AR-vertex if no pair of the initial three vertices sum to $c$. 
\end{observation}
\begin{proof}
    Two element sets with $c$ form an odd number which cannot be equal to any singleton. The three-element sum of the first three numbers is odd. Hence the lemma. 
\end{proof}
\begin{observation}\label{3+odd}
    From Lemma \ref{max10}, if $\{a_1, a_2, a_3\}$ is a set of distinct odd numbers with $a_1 < a_2 < a_3$ and $a_4$ is another odd number with $a_4 > a_3$, then a vertex with edge labels incident being $\{a_1, a_2, a_3, a_4\}$ will be an AR-vertex if  $a_1 + a_2 + a_3 \neq a_4$ and $a_1 + a_4 \neq a_2 + a_3$. 
\end{observation}
\begin{proof}
    An odd sum could repeat only when the sum of a three-element set equals the singleton. Two-element sums could be equal only one way because of the restriction imposed by the ascending order of elements. 
\end{proof}
\begin{observation}\label{3ap+new}
    From Lemma \ref{max10}, if $a_1, a_2$ and $a_3$ are in arithmetic progression with common difference $d$ with $a_1 > d$, then a vertex with edge labels incident being $\{a_1, a_2, a_3, a_4\}$, $a_4 \notin \{a_1, a_2, a_3\}$, is an AR-vertex if $a_4 \notin \{d, 2d, a_1 - d, a_3 + d, a_1 + a_2, a_1 + a_3, a_2 + a_3, a_1 + a_2 + a_3\}$. 
\end{observation}
\begin{proof}
    The set actually exhausts all possible scenarios where the new number could create two distinct subsets with the same subset sum. 
\end{proof}
\begin{lemma}\label{bsum}
    
     If $\{a_1, a_2, \ldots, a_j\}$ is the set of edge labels incident on an AR-vertex, then any vertex with degree $(j+1)$ and the set of edge labels incident being $\{a_1, a_2,\ldots, a_j, b\}$ is an AR-vertex if $b > a_1 + a_2 + a_3 +\ldots+ a_j$.
  \end{lemma}
\begin{proof}
    The largest number appearing as a linear combination of the first $j$ edge labels is the sum of all $j$ labels which is less than $b$ by choice. Every linear combination containing the edge label $b$ will be greater than $b$. Hence the lemma.  
\end{proof}

\section{Infinite Classes of AR-graphs.}

Whenever a new labeling is introduced, identifying graph classes which satify such a labeling is an important question to be addressed. One should keep in mind that there are graph labelings like Leech labeling \cite{Lee} which was initially defined in 1975 only for trees and till date there are only five known Leech trees. Still its extension to the class of all graphs provides infinite family of Leech graphs \cite{See}. In this section, we provide some infinite families of graphs which are AR-graphs.\\
 
{\bf \noindent (1) Paths and cycles.} \\
 For every injective edge labeling of G, Lemma \ref{1or2} assures vertices of a path or cycle are AR. Let the edges be labelled from 1 to $m$ in any order, paths and cycles are AR-graphs.\\
 
{\bf \noindent (2) Attaching a path to a cycle.}  \\
The only vertex that raises a concern would be the vertex where path is connected to the cycle. Let $v$ be that vertex. Since $v$ has degree 3, $v$ is not trivially AR. Also, the smallest such graph has 4 edges. Label the edges incident on $v$ as 1, 2 and 4 and label the remaining edges using $\{3, 5,6,7,\dots, m\}$. $G$ will be AR under this labeling which is from $E\rightarrow\{1,2,3,\dots, m\}$. Hence $G$ is an AR-graph. \\

{\bf \noindent (3) Attaching a path to a pendant vertex in an AR-graph.}\\
Direct consequence of Lemma \ref{1or2} and the fact that the initial graph is an AR-graph.\\

{\bf \noindent (4) Attaching a path of length at least $2$ to a degree $2$ vertex in an AR-graph.}\\
Since the initial graph $G$ has an AR-labeling $f:E\rightarrow \{1,2,3,\dots,m(G)\}$ and the path to be attached has at least two edges, we can use edge labels $m+1$ and $m+2$ while labeling the new graph. Let $v$ be the vertex on which the new path is attached. Keeping the edge labels of $G$ from $f$, let the two edges in $G$ incident on $v$ be some $x$ and $y$. Label the third edge incident on $v$ as $m+1$ if $x+y\neq m+1$, otherwise label the edge as $m+2$. By Lemma \ref{3lab}, $v$ is an AR-vertex. Hence the result.\\

{\bf \noindent (5) Attaching a path to every vertex of a cycle.}\\
    Let $C_{n_1}$ be the cycle with $n_1$ vertices. Let $G$ be a graph formed by attaching paths to all vertices of $G$.  Suppose $m\geq 4{n_1} - 1$. Label all the edges incident on the vertices of the cycle using consecutive odd numbers starting from 1. By Lemma \ref{3odd}, every vertex on the cycle will be AR-vertices. Use the labels $\{2,4,6,\dots, 4{n_1}-2, 4{n_1}, 4{n_1}+2, \ldots, m\}$ on the remaining edges. Applying Lemma \ref{1or2}, this labeling is AR and since it is from $E \rightarrow \{1,2,3,\dots, m\}$, $G$ is an AR-graph.\\
    Suppose $m \leq 4{n_1}-2$. Let the cycle be $v_{1}v_{2}v_{3}\dots v_{n_1}$. Label the edges on the cycle using 1 to $n_1$ ($v_1 v_2$ as 1, $v_2 v_3$ as 2, $\dots$, $v_{n_1} v_1$ as $n_1$). Now, label the edges not in the cycle but incident on its vertices as follows. Label the edge incident on $v_i$ as $n_1 + i - 1$ for $i = 2,3,4, \dots n_1-1 $ and the edge incident on $v_{n_1}$ as $2n_1$ and edge incident on $v_1$ as $2n_1-1$. Label the remaining edges using numbers less than or equal to $m$, and not yet used in the labeling.\\
    The labels of edges incident on vertex $v_k$ on the cycle are $k - 1, k$ and $n_1 + k  - 1$ if $1 < k  < n_1$. $n_1$ - 1, $n_1$ and $2n_1$ are the labels of edges incident on $v_{n_1}$. 1, $n_1$ and $2n_1$ are the labels of edges incident on $v_1$. 
Due to Lemma \ref{1or2}, the vertices not in the cycle are trivially AR-vertices.
    By Lemma \ref{3lab}, the vertex $v_k$ on the cycle is $AR$ since $k-1 + k  \neq n_1 + k  -1$, if $k  < n_1$ and $v_1$ is $AR$ since $n_1 + 1 \neq 2n_1$, as a cycle has more than two vertices. And $v_{n_1}$ is an AR-vertex since $n_1 - 1 + n_1 \neq 2n_1$. So, attaching a path to every vertex of a cycle results in an AR-graph.\par
 \begin{figure}[ht]\center
  \includegraphics[scale=0.9]{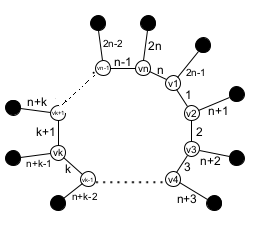} \\
  \caption{AR-labeling of attaching path of length 1 to $C_n$}\label{Figure 1}
\end{figure}
    Figure \ref{Figure 1} illustrates the AR-labeling of attaching a path of length 1 to every vertex of a cycle $C_n$. By Lemma \ref{1or2}, this labeling can be extended to a graph obtained by attaching a path of arbitrary length to every vertex of $C_n$.\\

{\bf \noindent (6) The Cartesian product of cycles with $K_2$, $C_n \square K_2$.}\\
We have two copies of the same cycle with $n$ vertices and each vertex in a cycle
is adjacent to the corresponding vertex in the other cycle.\\
So if  $v_1 v_2 v_3 \dots v_n v_1$ is the first cycle and  $u_1 u_2 u_3 \dots u_n u_1$ is the second cycle, the additional
edges would be $v_1 u_1 , v_2 u_2, \dots , v_n u_n$.\\
Consider the labeling $f$, $f(v_iv_{i+1}) = i, f(u_iu_{i+1}) = n+i$ for $1\leq i \leq n-1$,\\ $f(v_nv_{1}) = n, f(u_nu_{1}) = 2n$  and $f(u_iv_i) = 2n+i$ for every $ i$. \\
It is easy to verify that $f$ is an AR-labeling of $C_n \square K_2$, which makes it an AR-graph. In fact, we have a stronger result that every Hamiltonian cubic graph is an AR-graph.

\section{Cubic AR-graphs}

 A cubic graph is a 3-regular graph, that is, a graph in which the degree of each vertex is 3. Hence a cubic graph $G$ on $n$ vertices has $\frac{3n}{2}$ edges. Moreover, labeling each edge of $G$ using an odd number, we get an AR-labeling of $G$ by Lemma \ref{3odd}. So, there exists an AR-labeling $f:E(G)\rightarrow \{1,2,3,\dots,3n-1\}$. Also if $G$ is cubic, then $G$ has even number of vertices. That is, $n = 2k$ for
some $k \in \mathbb{N}$. $K_4$ is the cubic graph on 4 vertices and the labeling in Figure \ref{K4} shows that $K_4$ is an AR-graph.\par
\begin{figure}[ht]
    \centering
    \includegraphics[scale=0.65]{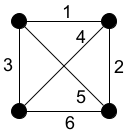}
    \caption{AR-labeling of $K_4$}
    \label{K4}
\end{figure}

There are exactly 2 cubic graphs on 6 vertices, one is edge-transitive and the other is
not. Figure \ref{cubic on 6 vertices} shows that both these cubic graphs are AR-graphs.
\begin{figure}[ht]
    \centering
    \includegraphics[scale=0.65]{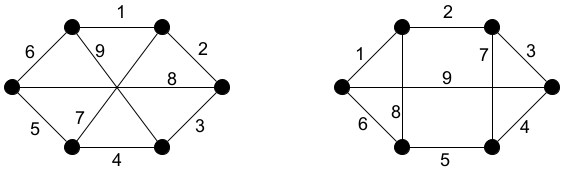}
    \caption{AR-labeling of Cubic Graphs with Order 6}
    \label{cubic on 6 vertices}
\end{figure}

\begin{theorem}
    All Hamiltonian cubic graphs are AR-graphs.
\end{theorem}
    
\begin{proof}
    Let $G$ be a cubic graph with more than 6 vertices (since 4 and 6 are settled).     For a cubic graph with $n = 2k $ vertices, there will be $3k$ edges. Being Hamiltonian,
there exists a cycle $C_n$ in $G$ which includes all the vertices of $G$.  Let the cycle $C_n$ be $v_1 v_2 v_3 \dots v_n v_1$. Label the edges of the cycle from 1 to $n$, $v_1 v_2$ as 1, $v_2 v_3$ as 2, $\dots$, $v_j v_{j+1}$ as $j$, $\dots$, $v_{n-1} v_n$ as $n-1$ and $v_n v_1$ as $n$. Now we have to label
the remaining edges using numbers $n+1, n+2,\dots, n + k$.  Since we have labelled edges
continuously using consecutive natural numbers, the only sum that repeats would be $n+1$ which appears as the sum of
edge weights for $v_1$ as well as $v_{k+1}$.\par
For each vertex $u$ in G, consider the sum of weights of edges incident on $u$ as a 
label. So now we have $k$ distinct numbers and $2k$ vertices labelled with only one
label repeating exactly once. Since $G$ has more than 6 vertices, we have $k \geq 4$.
Applying $Observation$ \ref{klab}, there exists an AR-labeling of $G$ (Edge labels
given are distinct from the sum of weights of edges already incident on those
vertices). Since we have used only numbers from 1 to $3k$, $G$ is an AR-graph.
\end{proof}
\begin{theorem}
     Cubic graphs whose vertex set could be expressed as a disjoint union of vertex
sets of cycles in the graph are AR-graphs.
\end{theorem}

\begin{proof}
    Let $G$ be a cubic graph whose vertex set could be written as a disjoint union of
vertex sets of cycles in $G$.  Let $C^1 , C^2 , \dots , C^b$ be the $b$ disjoint cycles. Let cycle $C^i$ has $a_i$ vertices, then the disjoint cycles are of the form $C_{a_1}$, $C_{a_2}$, $\dots$, $C_{a_b}$. Clearly, $\sum_{i=1}^{b} a_i = n$.
Now label the edges in $C_{a_1}$ from 1 to $a_1$, edges in $C_{a_2}$ from $a_1 + 1$ to $a_1 + a_2$, edges in $C_{a_3}$ from $a_1 + a_2 + 1$ to $a_1 + a_2 +a_3$, proceeding this way, edges in $C_{a_j}$ from $\sum_{i=1}^{j-1} a_i +1$ to $\sum_{i=1}^{j} a_i$, eventually edges in $C_{a_b}$ from $\sum_{i=1}^{b-1} a_i +1$ to $\sum_{i=1}^{b} a_i = n$.\par
We have to label the remaining $k = \frac{n}{2}$ edges with numbers from $n+1$ to $n+k$. Since cubic graphs with less than or equal to 6 vertices are AR graphs, let us consider only those graphs having more than 6 vertices.\par
Since we have labelled the edges continuously till this point, after having labelled $n$
edges, a number appears as the sum of edge weights in exactly one cycle and at most
twice. Since there are at least 8 distinct vertices and $G$ is satisfying the condition of $Observation$ \ref{klab}, we conclude that there exists a labeling of those $k$  edges that results in an AR-labeling of $G$ from $E(G)$ to $\{1,2, \ldots,3k\}$, hence $G$ is an AR-graph. 
\end{proof}
\begin{coro*}\label{maxdeg}
    A graph $G$ with $\Delta(G) \leq 3$, whose vertex set could be partitioned into vertex sets of cycles of $G$, is an AR-graph.
\end{coro*}
 \begin{proof}
     Since $C_3, C_4$, $K_4 - e$ and $K_4$ are AR-graphs (AR-labeling of $K_4 - e$ can be deduced from Figure \ref{K4}), we need to consider only graphs with $n>4$.
     The labeling technique follows from the theorem. Label all the edges in the cycles using consecutive natural numbers starting from 1. While labeling the remaining edges, if there are at least four of them, the corollary follows from $Observation$ \ref{klab}. Assume $m \leq n+3$. Under the purview of this corollary, $m=n$ implies a cycle which is an AR-graph. Suppose $m = n+1$, replacing the label 2 with $n+1$ results in an AR-labeling irrespective of where the edge label 2 is placed since $n>4$.  Suppose $m = n+2$, replace the labels 2 and 3 with $n+1$ and $n+2$ respectively. The only vertex that raises a concern is the one on which the edge with label $n+2$ is incident along with another edge with label $x<n$. Labeling the possible remaining edge incident on that vertex using label 2 results in an AR-labeling of the graph. Similarly, if $m = n+3$,  replace the labels 2, 3 and 4 with $n+1$, $n+2$ and $n+3$ respectively. The only vertex that raises concern is the one on which the edge having label $n+3$ is incident along with another edge with label $x\leq n$. Labeling the possible remaining edge on that vertex as 2 gives an AR-labeling. Hence the result. 
    \end{proof}

\section{Sierpinski Graphs and Sierpinski Gasket Graphs}
The Sierpi$\Acute{\text{n}}$ski graphs are typically used in complicated frameworks, fractals and recursive assemblages. Different networks associated with these graphs have applications in computer science, physics and chemistry \cite{Muh}. The Sierpi$\Acute{\text{n}}$ski graphs $S(n,3)$ are defined in the following way \cite{San}:
\begin{center}
    $V(S(n,3)) = \{1,2,3\}^n$, 
\end{center}
two different vertices $u= (u_1,u_2,\ldots,u_n)$ and $v = (v_1,v_2,\ldots,v_n)$ being adjacent if and only if there exists an $h \in \{1,2,\ldots,n\}$ such that 
\begin{enumerate}
    \item $u_t = v_t$, for $t=1,2,\ldots,h-1$;
    \item $u_h \neq v_h$; and 
    \item $u_t = v_h$ and $v_t = u_h$ for $t = h+1, h+2, \ldots, n$
\end{enumerate}

All Sierpi$\Acute{\text{n}}$ski graphs have maximum degree 3 and hence by Corollary \ref{maxdeg}, all of them are AR-graphs. Though the corollary proves the existence of an AR-labeling which makes $S(n,3)$ an AR-graph, the actual labeling is not obtained. In the proof of the following theorem, we give an explicit labeling for $S(n,3)$. 
\begin{theorem}
    The Sierpe$\Acute{\text{n}}$ski graph $S(n,3)$ is an AR-graph, for every $ n \in \mathbb{N}$.
\end{theorem}

\begin{proof}
    $S(1,3)$, being a triangle is trivially an AR-graph. The edge labeling of the induced $S(2,3)$ in the upper half of Figure \ref{Sierpenski Graphs} shows that $S(2,3)$ is an AR-graph. From $S(2,3)$, we are giving an iterative method as an explicit labeling technique for Sierpi$\Acute{\text{n}}$ski graph with $n>2$. The labeling of $S(3,3)$ is given as an example in Figure \ref{Sierpenski Graphs}.
    
\begin{figure}[ht]
    \centering
    \includegraphics[scale=0.35]{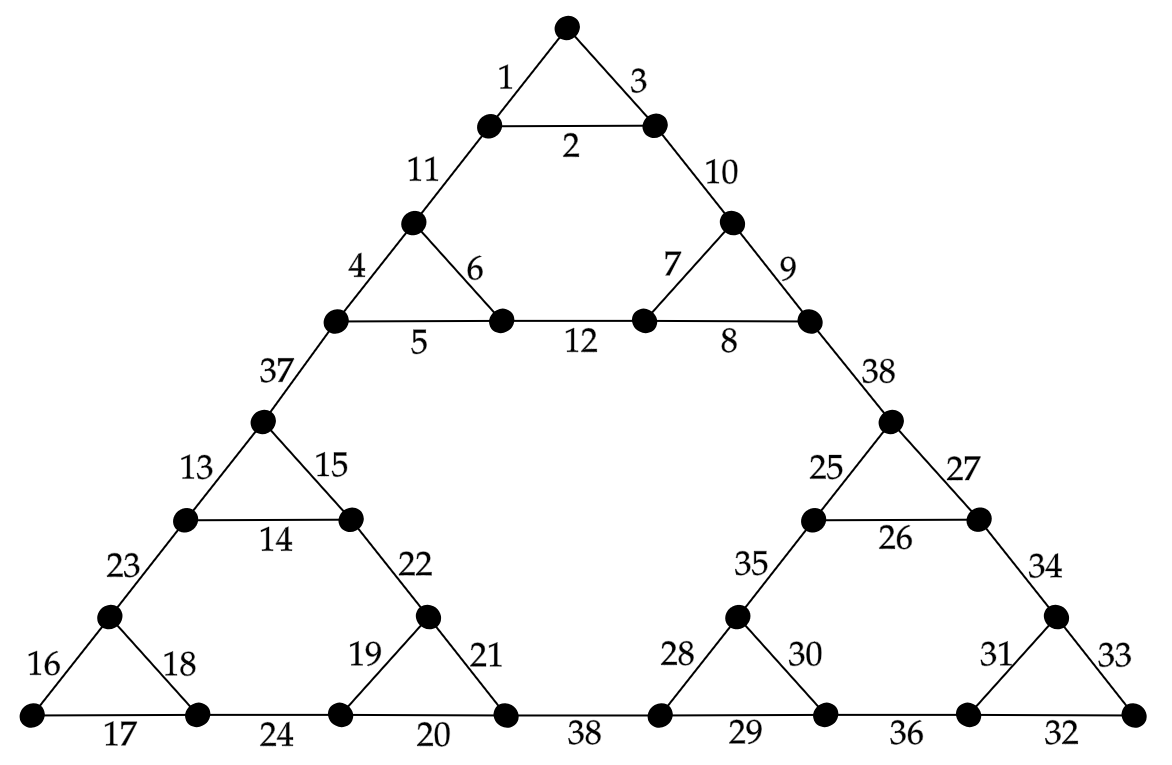}
    \caption{AR-labeling of $S(3,3)$}
    \label{Sierpenski Graphs}
\end{figure}

For $n>3$, the iterative labeling technique for $S(n,3)$ starts using the labels of $S(3,3)$ in Figure \ref{Sierpenski Graphs}. Since $S(n,3)$ is made up of three copies of $S(n-1,3)$ along with three edges which connect them, we call the three copies, M-Block, L-Block and R-Block. We label the edges in the M-Block using the exact labels from $S(n-1,3)$. The edge in L-Block corresponding to the edge labelled $k$ in $S(n-1,3)$ is labelled $\alpha +k$, where $\alpha$ denotes the number of edges in $S(n-1,3)$. Similarly, the edge in R-Block corresponding to the edge labelled $k$ in $S(n-1,3)$ is labelled $2\alpha +k$. There are three more edges left to be labelled. Let them be $e_{ml}, e_{mr}$ and $e_{lr}$ where the suffices denote the blocks on which the edge is incident on. Let $e_{ml}$ be incident on the vertices $v_l^m$ and $v_m^l$, where $v_l^m$ denotes the vertex in L-Block adjacent to $e_{ml}$ and $v_m^l$ represents the vertex in M-Block adjacent to $e_{ml}$. Similarly, let $e_{lr}$ be incident on $v_l^r$ and $v_r^l$, and $e_{mr}$ be incident on $v_r^m$ and $v_m^r$. Label the edge $e_{ml}$ with $3\alpha + 1$. Now, if $n$ is even, $e_{lr}$ is labelled $3\alpha + 3$ and $e_{mr}$ is labelled $3\alpha + 2$. If $n$ is odd, $e_{lr}$ is labelled $3\alpha + 2$ and $e_{mr}$ is labelled $3\alpha + 3$. 

We claim that the above labeling is an AR-labeling. A restriction of Lemma \ref{kshift} to three vertices shows that all vertices other than the vertices adjacent to edges $e_{lr}$, $e_{mr}$ and $e_{ml}$ are AR-vertices. Now consider $v_m^l$ and $v_m^r$, the set of edge labels incident on them is of the form $\{k, k+1, 3\alpha + c\}$ where $c$ is a non-negative integer and $k \in \mathbb{N}, k < \alpha$. By Lemma \ref{3lab}, $v_m^l$ and $v_m^r$ are AR-vertices. For $v_l^m$, the set of edge labels incident is of the form $\{\alpha + 1, \alpha + 2, 3\alpha + 1\}$ and for $v_r^m$, the set of edge labels incident is of the form $\{2\alpha + 1, 2\alpha + 2, 3\alpha + c\}$ where $c \in \{2,3\}$. By Lemma \ref{3lab}, $v_l^m$ and $v_r^m$ are also AR-vertices. Before considering $v_l^r$ and $v_r^l$, the remaining two vertices, we observe that the number of edges in $S(n,3)$ is odd if and only if $n$ is odd. $S(1,3)$ has 3 edges, $S(2,3)$ has $3(3)+3 = 12$ edges, $S(3,3)$ has $3(12)+3 = 39$ edges. $|E(S(n,3))| = 3|E(S(n-1,3))| + 3$, means the number of edges in $S(n,3)$ alternatively becomes even and odd. Using this observation, we can see that the edge $e_{lr}$ is always labelled an even number in our specific labeling mentioned above. Since the labels of edges incident on $v_l^r$ as well as $v_r^l$ are consecutive numbers, Lemma \ref{3lab} ensures that both vertices are AR-vertices. Hence, the result.  
\end{proof}

The Sierpi$\Acute{\text{n}}$ski gasket graphs $S_n, n \geq 1$, can be defined geometrically as the graph whose vertices are the intersection points of the line segments of the finite Sierpi$\Acute{\text{n}}$ski gasket $\sigma_n$ and line segments of the gasket as edges\cite{San}. $S_1$, being a triangle, is trivially an AR-graph. Figure \ref{Sierpenski Gasket Graphs} shows that $S_2$ and $S_3$ are also AR-graphs.\\

\begin{figure}[ht]
    \centering
    \includegraphics[scale=0.5]{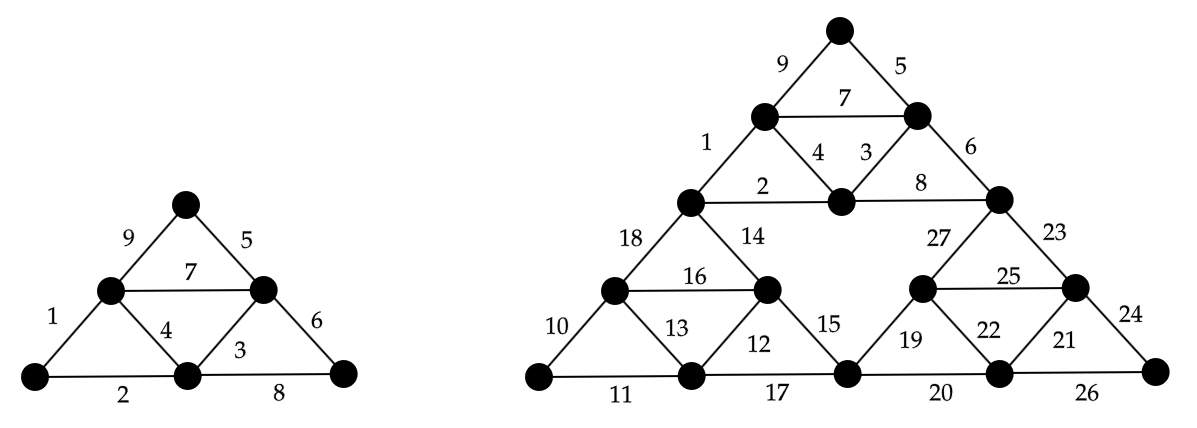}
    \caption{AR-labeling of Sierpi$\Acute{\text{n}}$ski Gasket Graphs $S_2$ and $S_3$}
    \label{Sierpenski Gasket Graphs}
\end{figure}

\begin{theorem}
 The Sierpi$\Acute{\text{n}}$ski Gasket Graph $S_n$ is an AR-graph, for every $ n \in \mathbb{N}$. 
\end{theorem}

\begin{proof}
    
We label each graph inductively using the labels from the preceding graph. Let us identify the Sierpi$\Acute{\text{n}}$ski Gasket Graph $S_n$ with $n > 2$ as having three blocks of $S_{n-1}$ in which each pair of blocks share one common vertex. We call those blocks, the M-block, L-Block and R-Block (Middle, Left and Right). The vertex common to the blocks are named $v_{lm}$, $v_{mr}$ and $v_{lr}$, the suffices denoting the blocks that share the vertex. 

We label the $3^n$ edges of $S_n$ thus. The M-Block is labelled identically to our labeling of $S_{n-1}$. Every edge gets the same label from $S_{n-1}$. The edges of L-Block are labelled as $3^{n-1} + w(e)$ where $w(e)$ denotes the label of the corresponding edge in $S_{n-1}$. The edges of R-Block are labelled as $ 2 \times 3^{n-1} + w(e)$ where $w(e)$ denotes the label of the corresponding edge in $S_{n-1}$.

Let us consider labeling $S_n$ using the labels we gave for $S_3$. From Lemma \ref{kshift}, we have all the vertices of $S_n$ other than $v_{lm}$, $v_{mr}$ and $v_{lr}$ as AR-vertices. Now consider $v_{lm}$, the edges incident are labelled $3^{n-2} + 1, 3^{n-2} + 2, 3^{n-1} + 5$ and $3^{n-1} + 9$. Since $n \geq 4$ and $1+9 \neq 2+5$, by Lemma \ref{max10}, $v_{lm}$ is an AR-vertex. The edges incident on $v_{mr}$ are labelled $3^{n-2} + 6, 3^{n-2} + 8, 2 \times 3^{n-1} + 5$ and $2 \times 3^{n-1} + 9$.  Since $n \geq 4$ and $1+9 \neq 2+6$, by Lemma \ref{max10}, $v_{mr}$ is an AR-vertex. For $v_{lm}$, the edges incident are labelled $3^{n-1} + 6, 3^{n-1} + 8, 2 \times 3^{n-1} + 1$ and $2 \times 3^{n-1} + 2$. Since $6+8 > 2$ and $6 + 2 \neq 8 + 1$, by Lemma \ref{max10}, $v_{lm}$ is also an AR-vertex. Hence the proof. 
\end{proof}

\section{Perfect Binary and Ternary Trees}
A binary tree is a tree structure in which each node has at most two children, referred to as the left child and the right child. A perfect binary tree $T_{r,2}$ is a binary tree in which all interior nodes have two children and all leaves have the same depth $r$ (the depth of a node $v$ is defined as the number of edges from the root node to $v$) \cite{Yum}. Similarly, A perfect $t-$ary tree $T_{r,t}$ is a $t-ary$ tree in which all interior nodes have $t$ children and all leaves have the same depth $r$ \cite{Yum}. The perfect $3-$ary tree $T_{r,3}$ is also called perfect ternary tree.
\par For $t\geq2$ and $r\geq1$, a glued $t-$ary tree $GT(r,t)$ is obtained from two copies of a perfect $t-$ary tree by pairwise identification of their leaves. The vertices obtained by identification are called quasi-leaves of $GT(r,t)$. The glued binary trees $GT(r,2)$ are used in quantum computing and can be used to solve problems exponentially faster than classical algorithms \cite{Chi}\cite{Zi}. The concept of glued $t-$ary trees can be generalized by gluing $n$ perfect $t-$ary trees, instead of two \cite{Dha}. More precisely, the $n^{th}$ generalized glued $t-$ary tree $GT_{r,t}^{(n)}$ of depth $r$ is obtained from $n$ copies of $T_{r,t}$ by identifying their leaves.  
\begin{theorem}
    The perfect binary trees $T_{r,2}$ are AR-graphs, for every $r \in \mathbb{N}$.
\end{theorem}

\begin{proof}
    Consider an arbitrary perfect binary tree with depth $r$. Label the pendant edges from one end, say from the leftmost leaf, continuously using numbers 1 to $2^r$. After that, label the remaining edges incident on the supporting vertices(edges in the immediate next level) continuously from $2^r + 1$ to $2^r + 2^{r-1}$ starting from the rightmost edge incident on the supporting vertex. Once all the edges in this level are labelled, continue labeling the next level of edges, starting from the leftmost edge. The process is continued till all the edges are labelled, with the first edge to be labelled after each level is alternating between the leftmost and rightmost ones. It is like labeling the edges from 1 to $m$ in a continuous fashion where each new higher level of edges is labelled starting from the edge which is adjacent to the immediate previous pair of edges labelled. This is an AR-labeling of a perfect binary tree when the depth $r \neq 4a+1, a\in \mathbb{N}$. An illustration of the labeling technique is given for $T_{3,2}$ in Figure \ref{Binary Tree}.

\begin{figure}[ht]
    \centering
    \includegraphics[scale = 0.48]{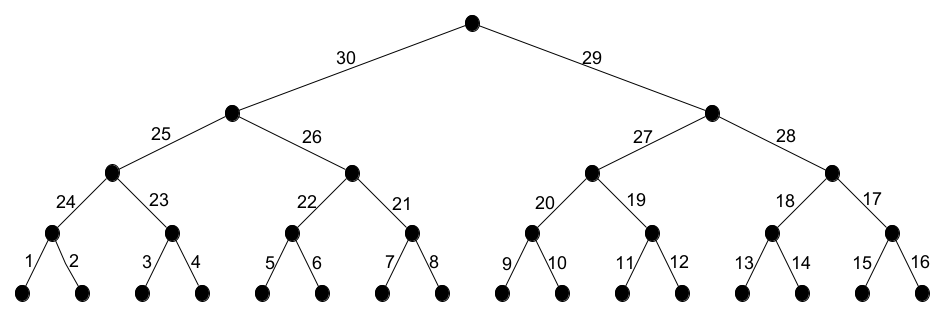}
    \caption{AR-labeling of $T_{3,2}$}
    \label{Binary Tree}
\end{figure}

    The origin vertex with degree 2, as well as the leaves, are trivially AR-vertices by Lemma \ref{1or2}. All the other vertices have the set of edge labels incident on them to be $\{k, k+1, \alpha\}$ for some odd $k \in \mathbb{N}$  and $\alpha \in \mathbb{N}$. For an arbitrary non-support vertex $u$ of degree three, we can write $k = \beta + k'$ where $\beta = 2^r + 2^{r-1} \ldots + 2^j$ for some $j \in \mathbb{N}, j < r-1$ and $k' < 2^{j-1}$, Now $\alpha$ being an edge label in the immediate higher level of edges labelled $k$ and $k+1$, it will be bounded by the highest edge label assigned in that level, that is $\alpha \leq \beta + 2^{j-1} + 2^{j-2}$. But then $2k + 1 > \alpha$. Hence by Lemma \ref{3lab}, $u$ is an AR-vertex.
    
The only vertices left are the support vertices of the graph. Let $v$ be an arbitrary support vertex of the binary tree, the set of edges labels incident on $v$ are $\{k, k+1, \alpha\}$ for some odd $k \in \mathbb{N}$  and $\alpha \in \mathbb{N}$ where $\alpha = 2^r + 2^{r-1} - \frac{k+1}{2}$ + 1. By Lemma \ref{3lab}, $v$ is an AR-vertex if $2k + 1 \neq \alpha$. Assuming the equality and simplifying would lead us to the equality $2^r + 2^{r+1} = 5k + 1$. This is true only when $r = 4a+1, a \in \mathbb{N}$. So when $r$ is not of the form $4a+1, a \in \mathbb{N}$, the above labeling would be an AR-labeling of the perfect binary tree with depth $r$. 
    
    Now let us consider the number of levels $r$ to be of the form $4a+1, a \in \mathbb{N}$. For example, if $r = 5$, we have $k= 19$ such that the above labeling would result in a supporting vertex with incident edge labels \{19, 20, 39\} which is not an AR-vertex by Lemma \ref{3lab}. Similarly, when $r = 9$, the labeling would give rise to a supporting vertex having incident edge labels \{307, 308, 615\}, which by the same Lemma is not an AR-vertex. So, for binary trees with number of levels $r$ = $4a+1, a \in \mathbb{N}$, we swap the edge labels of supporting vertices $\{v_1, v_2\}$, the former adjacent to edge labels $k$ and $k+1$, the latter adjacent to edge labels $k-1$ and $k-2$ where $k$ is such that $5k = 2^r + 2^{r+1} - 1$. Since the third edge label incident on $v_2$ is $2k + 2$ in this case, swapping the edge labels of non-pendant edges would result in $v_1$ and $v_2$ being $AR$ by Lemma \ref{3lab}
    since the sets of edge labels incident on them will be respectively $\{k, k+1, 2k+2\}$ and $\{k-1, k-2, 2k+1\}$. Hence perfect binary trees with depth $r = 4a+1, a \in \mathbb{N}$ are also AR-graphs.  
\end{proof}

\begin{theorem}
The generalized glued binary trees $GT_{r,2}^{(n)}$ are AR-graphs, for $n\leq4$ .
\end{theorem}

\begin{proof}
The glued binary tree $GT(1,2)$ corresponds to cycle $C_4$ which is trivially $AR$. $GT_{1,2}^{(3)}$ and $GT_{1,2}^{(4)}$ denote respectively complete bipartite graphs $K_{2,3}$ and $K_{2,4}$ which are shown to be AR-graphs in Figure \ref{K23 and K24}. So we only need to consider generalized glued binary trees with depth at least two from now on.
    Let us consider an arbitrary glued binary tree $GT(r,2)$ with $r>1$. We label the first copy of the binary tree from 1 to $\mu$ as in $Theorem~ 5$, where $\mu$ denotes the total number of edges in one copy of the binary tree. Label the edges of the second copy of the binary tree using edge labels $\mu + 1$ to $2\mu$ in the following way. Corresponding to the edge labelled $k$ in the first copy, label the edge in the second copy as $\mu + k$. This will be an AR-labeling of $GT(r,2)$ with $r>1$.
    \begin{figure}[ht]
        \centering
        \includegraphics[scale=0.5]{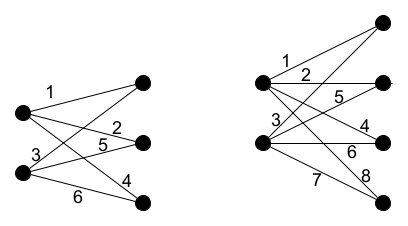}
        \caption{AR-labeling of $K_{2,3}$ and $K_{2,4}$}
        \label{K23 and K24}
    \end{figure}
   
    The quasi leaves and two origin vertices having degree 2 are trivially AR-vertices by Lemma \ref{1or2}. The vertices in the first copy of the binary tree are all AR-vertices and since the set of edge labels incident on an arbitrary vertex $v$ of degree three in the second copy of the binary tree is $\{\mu + a_1, \mu + a_2, \mu + a_3\}$ in which $\{a_1, a_2, a_3\}$ is the set of edge labels incident on an AR-vertex in the first copy, a restriction of Lemma \ref{kshift} to three vertices would then make $v$ an AR-vertex. Hence all glued binary trees $GT(r,2)$ are AR-graphs. Figure \ref{Glued Binary Trees} demonstrates the labeling with the example of $GT(3,2)$. 
     \begin{figure}[ht]
        \centering
        \includegraphics[scale=0.65]{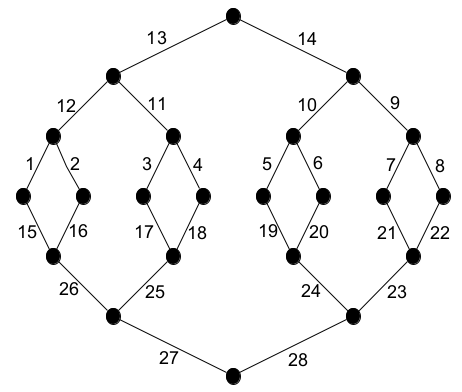}
        \caption{AR-labeling of $GT(3,2)$}
        \label{Glued Binary Trees}
    \end{figure}
    
    Moving onto generalized glued binary trees with three copies $GT_{r,2}^{(3)}$, the labeling of the first two copies follows from $GT(r,2)$ and in the third copy, corresponding to edges labelled $k$, the edges in the third copy are labelled $2\mu + k$. Vertices except the quasi leaves in the third copy of the binary tree are AR-vertices from the same argument we used while attaching the second copy. Hence, by Lemma \ref{1or2} and a restriction of Lemma \ref{kshift} to three vertices, all the vertices apart from the quasi leaves are AR-vertices. For all quasi leaves, the edge labels incident on them are $\{k, \mu + k, 2\mu + k\}$. These quasi leaves are hence AR-vertices by $Lemma~ 3$ since $k < \mu $. The labeling is an AR-labeling of $GT_{r,2}^{(3)}$.
    
    While considering the generalized glued binary tree with four copies $GT_{r,2}^{(4)}$, we label the first three copies of the binary tree except for the edges incident on the quasi leaves using the labels used for $GT_{r,2}^{(3)}$. Corresponding to edges labelled $k$, the edges other than those incident on the quasi leaves in the fourth copy, are labelled $3\mu + k$. All but the quasi leaves and support vertices of the fourth copy of the binary tree will be AR-vertices by Lemma \ref{1or2} and the restriction of Lemma \ref{kshift} to three vertices. We rotate cyclically the edge labels in the fourth copy of the binary tree corresponding to edges labelled 1, 4, 3 and 2 in the first copy of the binary tree. Now for all remaining edge labels from the supporting vertices of the fourth copy of the binary tree, we swap the corresponding labels of the pendant edges from the same supporting vertex in the fourth copy. That is, corresponding to the edge labelled $k$, $k$ being odd $k\neq 1, 3$, the edge in the fourth copy will be labelled $3\mu + k + 1$.  The edges corresponding to the edges labelled $l$ in the first copy, $l$ being even, $l\neq 2, 4$, are labelled $3\mu + l - 1$ in the fourth copy.
    
    Now the set of edge labels of quasi leaves can be classified into two sets and four remaining vertices. Vertices with set $A$ of edge labels $\{k, \mu + k, 2\mu + k, 3\mu + k + 1\}$, set $B$ of edge labels $\{l, \mu + l, 2\mu + l, 3\mu + l - 1\}$ with $k$ being odd and $l$ even. The four quasi leaves will have edge labels $\{1, \mu + 1, 2\mu + 1, 3\mu + 4\}$, $\{4, \mu + 4, 2\mu + 4, 3\mu + 3\}$, $\{3, \mu + 3, 2\mu + 3, 3\mu + 2\}$ and $\{2, \mu + 2, 2\mu + 2, 3\mu + 1\}$. Also, the two supporting vertices in the first copy of the binary tree will have sets of labels of incident edges to be $\{1,4,\gamma\}$ and $\{2,3,\gamma + 1\}$. Since $k \neq 1$, vertices in $A$ will be AR-vertices by $Observation$ \ref{3+even}, Vertices in $B$ will be AR-vertices by Lemma \ref{newodd}. The four quasi leaves are AR-vertices by $Observation$ \ref{3+even} and Lemmas \ref{ap}, \ref{newodd} and \ref{bsum}. The two supporting vertices will be AR-vertices if $r>2$ by Lemma \ref{3lab}, since $r>2$ implies $\gamma > 8$. Hence, the labeling is an AR-labeling of $GT_{r,2}^{(4)}$ with $r>2$. When $r=2$, interchange the labels $5$, which is $\gamma$ and $11$ which is $\mu + \gamma$. Since $11$ is incident on a degree two vertex and a degree three vertex having consecutive numbers, say $x$ and $x+1$ as labels of two incident edges with $x>6$, by Lemma \ref{1or2} and Lemma \ref{3lab}, swapping edge labels between $5$ and $11$ would result in an AR-labeling of $GT_{2,2}^{(4)}$. Hence the result.
    \end{proof}
\begin{theorem}
The perfect ternary trees $T_{r,3}$ are AR-graphs, for every $r \in \mathbb{N}$.
\end{theorem}

\begin{proof}
    Figure \ref{Ternary Trees} shows that perfect ternary tree with $r=3$ is an AR-graph. One can easily deduce from Figure \ref{Ternary Trees} that $T_{r,2}$ is also an AR-graph. This label could be extended to arbitrary $r$ levels by continuing the labeling with consecutive natural numbers. After each level, labeling of the next lower level commences from the leftmost edge of the level. There are four vertices say $v_1, v_2, v_3, v_4$ in the first three levels with sets of edge labels incident on them being respectively $\{1, 3, 5\}$, $\{1, 2, 4, 12\}$, $\{2, 13, 14, 17\}$ and $\{4, 15, 16, 18\}$. Now $v_1$ is an AR-vertex by Lemma \ref{3odd}, $v_2$ is an AR-vertex by Lemma \ref{bsum} and $v_3$ and $v_4$ can be verified to be AR-vertices by Lemma \ref{max10}. All the remaining vertices in the ternary tree have sets of edge labels of the form $\{a, k, k+1, k+2\}$ for some $k \in \mathbb{N}, k>a$. Since $2 < a < k-1$, by Observation \ref{3ap+new}, all those vertices are AR-vertices. Hence the result. 
\end{proof}
    \begin{figure}[ht]
        \includegraphics[scale=0.43]{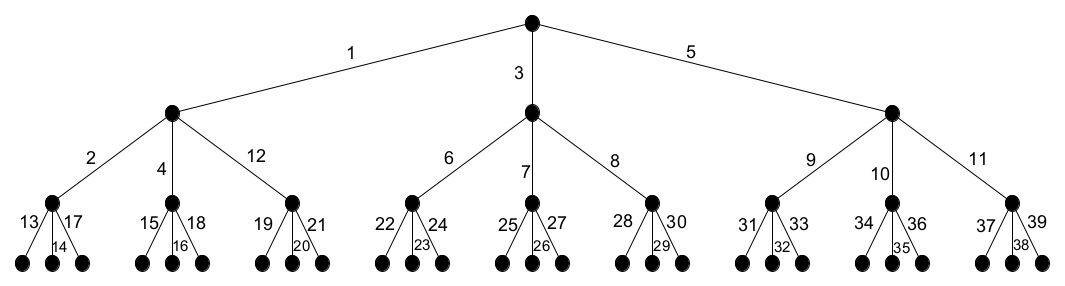}
        \caption{AR-labeling of $T_{3,3}$}
        \label{Ternary Trees}
    \end{figure}

\begin{theorem}
  The generalized glued ternary trees $GT_{r,3}^{(n)}$ are AR-graphs, for $n\leq4$.
\end{theorem}

 \begin{proof}

The glued ternary trees with depth one, $GT_{1,3}^{(n)}$ needs to be labelled distinctly from the rest. 
Figure \ref{Generalized glued ternary tree} is an AR-labeling of $GT_{1,3}^{(4)}$. Removing the origin vertex on which the the edges with maximum labels are incident will result in the  AR-labeling of the $GT_{1,3}^{(3)}$ and subsequent removal of another origin vertex with the same criterion will result in the AR-labeling of $GT_{1,3}^{(2)}$.  
\begin{figure}[ht]
    \centering
    \includegraphics[scale=0.6]{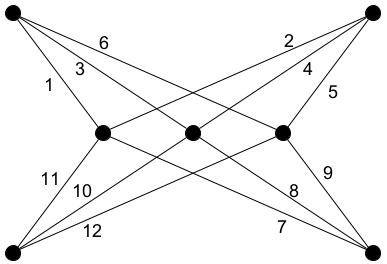}
    \caption{AR-labeling of $GT_{1,3}^{(4)}$}
    \label{Generalized glued ternary tree}
\end{figure}

While labeling all other glued ternary trees with two copies $GT(r,3)$, we can label the first copy of the ternary tree with the labeling technique used for labeling perfect ternary trees. The second copy is labelled in the following manner; Corresponding to every edge labelled $k$ in the first copy, the edge in the second copy is labelled $\alpha + k$ in the second copy where $\alpha$ denotes the number of edges in the first copy of the ternary tree. The quasi leaves will be AR-vertices due to Lemma \ref{1or2}. A restriction of Lemma \ref{kshift} to three vertices shows that the origin vertex of the second copy is an AR-vertex. A direct application of Lemma \ref{kshift} ensures that all other vertices in the second copy of the perfect ternary tree are AR-vertices. Figure \ref{Glued Ternary} demonstrates this labeling technique on $GT(2,3)$.\\ 
\begin{figure}[ht]
    \centering
    \includegraphics[scale=0.5]{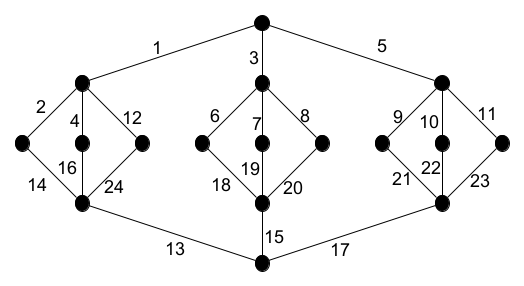}
    \caption{AR-labeling of $GT(2,3)$}
    \label{Glued Ternary}
\end{figure}

Moving onto generalized glued ternary trees with three copies $GT_{r,3}^{(3)}$, we can label the first two copies using the same labels from $GT(r,3)$. In the third copy of the ternary tree all the edges corresponding to edges labelled $k, k< \alpha - 1$ in the first copy are labelled $2\alpha + k$. The labeling of the final two edges is interchanged so that the vertices, say $v_1$ and $v_2$ on which they are incident on ends up being adjacent to sets of edge labels $\{\alpha, 2\alpha, 3\alpha - 1\}$ and $\{\alpha - 1, 2\alpha - 1, 3\alpha\}$. $v_1$ and $v_2$ are AR-vertices by Lemma \ref{3lab}. All the remaining vertices are AR-vertices by Lemma \ref{kshift} and its restriction to three vertices.

 While considering generalized glued ternary trees with four copies $GT_{r,3}^{(4)}$, the labels to edges incident on quasi leaves have to be given differently from the rest. Consider $GT_{r,3}^{(4)}$ with $r\geq4$. Let $v$ be an arbitrary supporting vertex in the first copy of the ternary tree.
Let $v_1$, $v_2$ and $v_3$ be the quasi leaves adjacent to $v$, who share the edges labelled $\beta, \beta - 1$ and $\beta - 2$ with $v$. The edges shared by these three leaves with $v', v''$ and $v'''$, the supporting vertices corresponding to $v$ in the second, third and fourth copies will be labelled in such a way that the set of edge labels incident on $v_1$ is $\{\beta, 2\beta, 3\beta - 1, 4\beta\}$, $v_2$ is $\{\beta - 1, 2\beta - 2, 3\beta - 2, 4\beta - 2\}$, $v_3$ is $\{\beta - 2, 2\beta - 1, 3\beta, 4\beta - 1\}$. Lemma \ref{max10} can be used to show that all the quasi leaves are AR-vertices under this labeling. Since the labeling hasn't changed the set of edge labels in vertices other than quasi leaves, all the remaining edges in the first three copies are labelled as in the case of $GT_{r,3}^{(3)}$ and the edges in the fourth copy in such a way that the edge corresponding to the edge labelled $k$ in the first copy is labelled $3\alpha + k$. Lemma \ref{kshift} and its restriction to three vertices ensures that remaining vertices in the fourth copy of the ternary tree are also AR-vertices. Hence the result is true for all $GT_{r,3}^{(4)}$ with $r\geq4$.

In $GT_{2,3}^{(4)}$, there exists a supporting vertex which has a non-consecutive set of edge labels incident. We proceed to label the edges in the three remaining copies in such a way that the set of edge labels incident on those three vertices will be $\{2,16,36,48\}$, $\{4,24,26,40\}$ and $\{12,14,28,38\}$. Since all other supporting vertices have consecutive numbers as labels of edges incident on them, we proceed to label the corresponding edges in the remaining copies the same way as we did when $r\geq 4$. This is an AR-labeling of $GT_{2,3}^{(4)}$.
In $GT_{3,3}^{(4)}$, there exist two supporting vertices which have non-consecutive sets of edge labels incident. we proceed to label the edges in the three remaining copies in such a way that the set of edge labels incident on those six vertices will be $\{13,52,91,131\}$, $\{14,53,92,134\}$, $\{17,56,95,130\}$, $\{15,54,93,133\}$, $\{16,55,94,135\}$ and $\{18,57,96,132\}$. Since all other supporting vertices have consecutive numbers as labels of edges incident on them, we use the same labeling technique mentioned above for other edges in the remaining copies of the perfect ternary tree. This is an AR-labeling of $GT_{3,3}^{(4)}$. So, the result is true for $GT_{r,3}^{(4)}$, for every $r \in \mathbb{N}$. Hence the Theorem.
 \end{proof}
\section{Concluding Remarks}
 Though this paper contains results about some cubic AR-graphs, whether all cubic graphs are AR-graphs is still an open question. In fact, identifying AR-graph classes seems to be an interesting venture. Can some parameterical bounds be obtained for AR-graphs will also be an interesting question. As a whole, this paper introduces the notion of AR-labeling and AR-labeling in general, and AR-graphs in particular, offers plenty of directions for research. 
\par Obviously, just like stars $K_{1,n}$ with $n>2$, there are infinitely many non AR-graphs, but the following theorem shows the existence of AR-labeling for every graph.
\begin{theorem}
    Given a graph G, there exists an AR-labeling of $G$.
\end{theorem}

\begin{proof}
  Let $G$ be a graph with edge set $E = \{e_1,e_2,\ldots,e_m\}$. Define $f:E \rightarrow \mathbb{N}$, $f(e_i)=2^{i-1}$. Since sums of distinct powers of 2 are always distinct, $f$ is an AR-labeling of $G$.
\end{proof}

Since every graph has an AR-labeling using sufficiently large natural numbers, the immediate optimization question would be to ask for the smallest natural number $k$ such that there exists an AR-labeling from the edge set of the graph to natural numbers from 1 to $k$. Hence, we define the notion of AR-index, which gives a rough idea of how close a graph is to being an AR-graph. 
\begin{defn}
    The minimum $k$ such that there exists an AR-labeling $f:E\rightarrow \{1,2,3,\dots,k\}$ is called the AR-index of G, denoted by $ARI(G)$.
\end{defn}
 Combining the restriction imposed by $m(G)$ and the edge labeling mentioned in Theorem 9, we have $m(G)\leq ARI (G) \leq 2^{m-1}.$ AR-graphs can be identified as graphs $G$ with $ARI(G) = m(G)$. Evaluating the AR-index of graph classes will be another possible area of research which can be extensively explored in the future.\\

\noindent \textbf{Acknowledgment:} The first author is supported by the Junior Research Fellowship (09/0239(17181)/2023-EMR-I) of CSIR (Council of Scientific and Industrial Research, India).


\begin{thebibliography}{30}


\bibitem{1} R. Balakrishnan and K. Ranganathan, A textbook of graph theory. Springer Science and Business Media, 2012.
\bibitem{Chi} A. M. Childs, et al. “An Example of the Difference Between Quantum and Classical Random Walks.” Quantum Information Processing 1 (2001): 35-43.
\bibitem{2} S. Costa, M. Dalai and S. D. Fiore. “Variations on the Erdős distinct-sums problem." Discrete Applied Mathematics 325 (2023): 172-185.

\bibitem{Muh} M. Iqbal and N. Alshammry. “Computer Architectures Empowered by Sierpinski Interconnection Networks utilizing an Optimization Assistant". Engineering, Technology and Applied Science Research. 14.14811-14818.10.48084/etasr.7572., 2024.
\bibitem{San} S. Klavžari. “Coloring Sierpi$\Acute{\text{n}}$ski Graphs and Sierpi$\Acute{\text{n}}$ski Gasket Graphs.” Taiwanese Journal of Mathematics, vol. 12, no. 2, 2008, pp. 513–22.
\bibitem{Lee} J. Leech. “Another tree labeling problem." Amer. Math. Month 82(9) (1975): 923–925.
\bibitem{Rob} A. Robertson and B. Landman, Ramsey Theory on the Integers (2nd edn.). AMS, 2014.
\bibitem{Dha}D. Roy, A. Lakshmanan and S. Klavžar. “Counting largest mutual-visibility and general position sets of glued t-ary trees". arXiv:2410.17611v1 [math.CO] 23 Oct 2024
\bibitem{See} S. Varghese, A. Lakshmanan and S. Arumugam,  “Leech index of a tree". J. Discrete Math. Sci. Cryptogr. 25(8) (2022): 2237–2247.
\bibitem{Yum} Yuming Zou and Paul E. Black, "Perfect binary tree", Dictionary of Algorithms and Data Structures [online], Paul E. Black, ed. 27 November 2019.
\bibitem{Zi} Zi-Yu Shi, Hao Tang, Zhen Feng, Yao Wang, Zhan-Ming Li, Jun Gao, Yi-Jun Chang, Tian-Yu Wang, Jian-Peng Dou, Zhe-Yong Zhang, Zhi-Qiang Jiao, Wen-Hao Zhou, and Xian-Min Jin, "Quantum fast hitting on glued trees mapped on a photonic chip," Optica 7, 613-618 (2020).
   
\end{thebibliography}
\end{document}